\documentclass[oneside]{amsart}
\usepackage{hyperref}

\usepackage{geometry} 
\usepackage{graphicx}           
\usepackage{amsmath}          
\usepackage{amsfonts}             
\usepackage{amsthm}                
\usepackage{dirtytalk}
\usepackage{amssymb}

\theoremstyle{remark}
\newtheorem{rem}{Remark}[section]
\theoremstyle{definition}
\newtheorem{defn}{Definition}[section]
\theoremstyle{plain}
\newtheorem*{thma*}{Theorem}
\newtheorem*{ack*}{Acknowledgements}
\newtheorem{thm}{Theorem}[section]
\newtheorem{lem}[thm]{Lemma}
\newtheorem{prop}[thm]{Proposition}
\newtheorem{cor}[thm]{Corollary}

\newtheorem*{prop*}{Proposition}
\newtheorem*{cla*}{Claim}

\begin{document}

\title{	Generic non-uniqueness of complete $H$-surfaces embedded in $\mathbb{H}^3$ \vspace{0em}}
\author{Cagri Haciyusufoglu}
\email{chaciyusufoglu@ku.edu.tr}
\address{Department of Mathematics, Koc University, 
Istanbul, Turkey}
\date{}

\maketitle  \vspace{-3em}
\begin{abstract}
We prove that, given $|H|<1$, a generic simple closed curve embedded in the asymptotic boundary of $\mathbb{H}^3$ (with respect to the supremum metric) bounds more than one complete  surface embedded in $\mathbb{H}^3$ which has constant mean curvature $H$. We remark that this is not true for the space of simple closed $\text{C}^1$-curves. \end{abstract}

\section{Introduction}

The classical tool for obtaining complete minimal surfaces embedded in  $\mathbb{H}^3$ with a given asymptotic boundary $\Gamma \subset \partial_\infty \mathbb{H}^3$ is to solve the Asymptotic Plateau Problem.  For a given $\Gamma$, this problem asks the existence of complete surfaces asymptotic to $\Gamma$ such that any compact part of some topological type minimizes area in a certain class of competitive surfaces  (For a survey on the topic see \cite{Co13-2}). In \cite{An82}, Anderson proved the existence of \emph{absolutely area minimizing surfaces} (Definition \ref{area}) in $\mathbb{H}^3$ . In fact, he proved the result for general dimensions and codimensions. In \cite{An83}, he also proved the existence for \emph{least area planes} in $\mathbb{H}^3$(Definition \ref{area}). These results imply that for a given simple closed curve $\Gamma \subset \mathbb{H}^3$ we  always have two complete embedded minimal surfaces with asymptotic boundary $\Gamma$ which might be identical. \smallskip

\noindent Concerning the uniqueness of solutions, it is known that a simple closed curve which is the boundary of a star-shaped domain in $\partial_\infty \mathbb{H}^3$ is the asymptotic boundary of a unique complete minimal surface in  $\mathbb{H}^3$(\cite{An82}, \cite{HL87}). On the other hand, there exist simple closed curves for which any absolutely area minimizing surface must have positive genus (\cite{An83}, \cite{Ha92}). In this paper we show that this is a generic property. Let $A_0$ denote the space of all simple closed curves with respect to the $\text{C}^0$-topology (supremum metric) and $A_1$ be the subset of $A_0$ consisting of curves for which any absolutely area minimizing surface has genus greater than one. Then the following holds:  \smallskip

\begin{prop*}
 $A_1$ is an open and dense subset of $A_0$.
\end{prop*}

\noindent From this and the existence of least area planes, it follows that generically a simple closed curve bounds more than one complete embedded minimal surface. An analogous result exists for the space of simple closed curves embedded in the boundary of mean convex 3-manifolds (\cite{TC15}) and here we follow those ideas.  \smallskip

\noindent To show that $A_1$ is open, we take a sequence of curves $\Gamma_n\in A_0\setminus A_1$ which converges to a simple closed curve $\Gamma$. Then for each $n$, $\Gamma_n$ is the asymptotic boundary of an absolutely area minimizing surface $\Sigma_n$ which has genus zero by definition. By classical arguments we can extract a subsequence of $\Sigma_n$ converging smoothly on compact sets to an absolutely area minimizing surface $\Sigma$ which has asymptotic boundary $\Gamma$ and has genus zero. This shows that $A_0\setminus A_1$ is closed so that  $A_1$ is open. This argument directly extends to the statement $A_g$ is open for every $g\in\mathbb{N}$, where $A_g$ is defined analogously (Remark \ref{genus}). This means that by a deformation of $\Gamma\in A_g$ which is sufficiently small compared to the diameter of $\Gamma$ we cannot obtain an absolutely area minimizing surface which has genus strictly less than $g$.  \smallskip

\noindent However, we can do the converse by applying a bridge principle \emph{at infinity}. In other words, we can \say{increase the genus} by a small deformation of the boundary curve $\Gamma$.  In \cite{MW13}, Martin and White developed a bridge principle for \emph{uniquely minimizing} surfaces (Definition \ref{area}) with smooth asymptotic boundary by attaching bridges to the asymptotic boundary curve and using it succesively they proved that any open surface can be properly embedded in $\mathbb{H}^3$. Using similar ideas we can show that $A_1$ is also dense in $A_0$.\smallskip

 \noindent Note that for our result we do not demand smoothness on the boundary curves. However, by the generic uniqueness of absolutely area minimizing surfaces (\cite{Co11}), in a $\text{C}^0$ neighborhood of a given simple closed curve, we can always find a smooth simple closed curve which bounds a unique absolutely area minimizing surface to which we can apply the bridge principle.  \smallskip

\noindent Analogous results exist for more general constant mean curvature surfaces. We will call a constant mean curvature surface with mean curvature $H\in \mathbb{R}$ as a $H$-\emph{surface}.  For $|H|<1$, complete embedded $H$-surfaces can be obtained by formulating the asymptotic Plateau problem for surfaces which minimizes area on compact domains with a volume constraint (\cite{To96}, \cite{AR97})(See Definition \ref{harea}). Following \cite{Co05} we will call such surfaces as \emph{H-minimizing} surfaces. Bridge principle for $H$-minimizing surfaces has been obtained in \cite{Co13-1}. Existence of \emph{H-minimizing planes} for $|H|<1$ is obtained in \cite{Co15} which also removes a gap in the Anderson's proof for the existence of least area planes. For this it is necessary to have one smooth point on the curve. In the most general form our main theorem can be stated as follows: 

\begin{thma*}
Given $|H|<1$, the set of simple closed curves $\Gamma\subset \partial_\infty \mathbb{H}^3$ with one smooth point and which bounds more than one complete $H$-surface embedded in $\mathbb{H}^3$ contains an open and dense subset of the space of all simple closed curves with respect to $\textnormal{C}^0$-topology.
\end{thma*}

\noindent In the next section we give the related background for our work. In the third section we first prove our result for complete embedded minimal surfaces and then we generalize it to $H$-surfaces. In the last section we will comment on some questions that naturally arises; what happens for the space of $\text{C}^1$ curves and for the curves with more than one component?
\begin{ack*}
\textnormal{I am very grateful to my advisor Baris Coskunuzer for his guidance and Francisco Martin for helpful conversations. Part of this work has been completed during my visit in University of Granada\footnote{This visit is funded by TUBITAK(The Scientific and Technological Research Council of Turkey) 2214 grant.}. I would like to thank the Department of Geometry and Topology for their hospitality.}
\end{ack*}
\section{Preliminaries}\label{prelim}

We will denote by $\mathbb{H}^3$ the 3-dimensional hyperbolic space. $\partial_\infty \mathbb{H}^3$ will denote the asymptotic boundary of $\mathbb{H}^3$ which is homeomorphic to the 2-dimensional sphere. If $A\subset  \mathbb{H}^3$, $\overline A$ will denote the closure of $A$ in the closed ball $ \mathbb{H}^3 \cup \partial_\infty \mathbb{H}^3$. We will call the set $\overline A \setminus A$ the \emph{asymptotic boundary of} $A$ which will be denoted by $\partial_\infty A$.

\begin{defn}\label{area}
A complete noncompact surface $\Sigma$ in $\mathbb{H}^3$ is said to be an  \emph{absolutely area minimizing} surface if any compact subsurface $S$ of $\Sigma$ minimizes area among all surfaces with boundary equal to $\partial S$. If for a simple closed curve $\Gamma \subset \partial_\infty \mathbb{H}^3$ there exists a unique absolutely area minimizing surface asymptotic to $\Gamma$, then we will call $\Gamma$ a \emph{uniquely minimizing curve} and $\Sigma$ a \emph{uniquely minimizing surface}. A complete embedded plane $P$ is called a \emph{least area plane} if any compact subdisk $D$ of $P$ minimizes area among all disks with boundary $\partial D$.

\end{defn}

\begin{lem} [\cite{An82}] \label{ex}
Let $p< n$ and $\Gamma^p$ be a closed immersed submanifold of $\partial_\infty \mathbb{H}^{n+1}$. Then there exists a complete absolutely area minimizing locally integral $p+1$ current $\Sigma^{p+1}\subset \mathbb{H}^{n+1}$ asymptotic to $\Gamma^p$.
\end{lem}

\begin{rem} \label{simp}
Later, in \cite{An83}, using the fact that any simple closed curve can be approximated by inscribed polygons it is proved that this result also holds for any simple closed curve in $\partial_\infty \mathbb{H}^3$.
\end{rem}

\noindent If we restrict ourselves to hypersurfaces, solutions are smooth and embedded away from a singular set of Hausdorff dimension $n-7$ by the interior regularity results of geometric measure theory. In the particular case of dimension $n=2$ the solutions are smooth embedded surfaces.  \smallskip

\noindent Note that in the construction of this solution, there is no control over the topology of the resulting absolutely area minimizing surface. In \cite{An83} Anderson solved the asymptotic Plateau problem for disk type (See also \cite{Ga97} and \cite{Co15}).

\begin{lem} \label{plane}
Let $\Gamma\subset \partial_\infty \mathbb{H}^3$ be a simple closed curve. Then there exists a complete embedded least area plane with asymptotic boundary $\Gamma$.
\end{lem}
 
\noindent Note that the solution given in the above theorem may not be absolutely area minimizing. In \cite{An83}  and \cite{Ha92}, they constructed examples of simple closed curves $\Gamma$ such that if $P$ is a least area plane asymptotic to $\Gamma$ then it cannot be absolutely area minimizing.   \smallskip

\noindent Let $A_0$ be the set of all simple closed curves embedded in $\partial_\infty \mathbb{H}^3$. The $\text{C}^0$-topology on $A_0$ is given as follows:

\begin{defn} \label{nhood}

Let $\Gamma\in \mathbb{H}^3$ be a simple closed curve. Then $\Gamma$ is the image of a continuous embedding $f:S^1\rightarrow \partial_\infty \mathbb{H}^3$. We define an open neighborhood $N_\epsilon(\Gamma)$ around $\Gamma$ as the set of all simple closed curves $\Gamma'$ for which there exists a one-to-one continuous map $f':S^1\rightarrow \partial_\infty \mathbb{H}^3$ with $f'(S^1)=\Gamma'$ such that $\|f(x)-f'(x)\|<\epsilon$ for all $x\in S^1$.

\end{defn}

\begin{rem}
 When $\epsilon>0$ is sufficiently small (relative to the diameter of $\Gamma$) another useful characterization of $N_\epsilon(\Gamma)$ can be given as follows: Let  $A_\epsilon(\Gamma) = \{x\in \partial_\infty \mathbb{H}^3: \text{dist} (x,\Gamma)<\epsilon\} $be an annulus around $\Gamma$. Then $N_\epsilon(\Gamma)$ consists of curves $\Gamma'\subset A_\epsilon(\Gamma)$ which are homotopic to $\Gamma$ in $A_\epsilon(\Gamma)$.
\end{rem}

\noindent The following lemma says that the set of uniquely minimizing smooth simple closed curves are dense. 
\begin{lem}[\cite{Co11}] \label{unique}
Let $\Gamma$ be a simple closed curve. Then there exists a sufficiently small $\epsilon>0$ such that, if $\{\Gamma_t:t\in(-\epsilon,\epsilon)\}$ is a set of simple closed curves foliating the annulus $A_\epsilon(\Gamma)$ around $\Gamma$, then except a countable subset of $(-\epsilon, \epsilon)$ $\Gamma_t$ bounds a unique absolutely area minimizing surface.
\end{lem}

\noindent We finish this section with the statement of a bridge principle \emph{at infinity} for uniquely area minimizing surfaces in $\mathbb{H}^{3}$. Let $\Gamma$ be a finite disjoint union of smooth simple closed curves and $\alpha$ be a line segment joining any two distinct points $p$,$q\in\Gamma$ such that $\alpha\cap \Gamma=\{p,q\}$ and $\alpha \perp \Gamma$ (Figure \ref{fig}). For $\epsilon>0$ define a neighborhood of the set $\Gamma\cup\alpha$ by $A_\epsilon(\Gamma\cup\alpha)=\{x\in \partial_\infty \mathbb{H}^3 : \text{dist}(x,\Gamma\cup\alpha)<\epsilon\}$. Note that $\Gamma$ separates the region into two (possibly disconnected) parts $\Omega^+$ and $\Omega^-$ with $\partial\Omega^+=\partial\Omega^-=\Gamma$. Let $\Omega^+$ be the region containing $\alpha$ and $\{\Gamma_t:0<t<\epsilon\}$ be a foliation of $\Omega^+\cap A_\epsilon(\Gamma\cup\alpha)$ such that $\Gamma_t$ converges to $\Gamma\cup\alpha$ as $t$ goes to $0$. Lastly, we call the closure of the set $P_\alpha:=\Omega^+ \cap A_\epsilon(\alpha)$ a \emph{bridge along} $\alpha$. The following theorem is obtained in \cite{MW13} (See also \cite{Co15}).

\begin{figure}[b]
\centering
\includegraphics[scale=0.75]{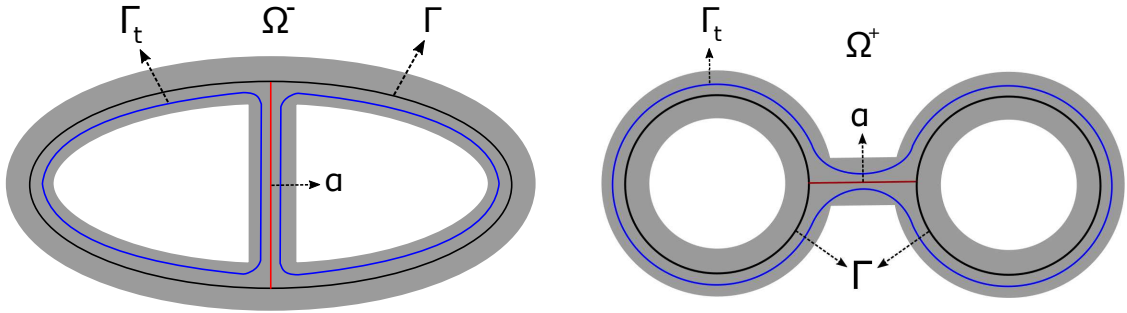}
\caption{$\Gamma$ is connected on the left and $\Omega^+$ is the region enclosed by $\Gamma$. On the right $\Gamma$ is disconnected and $\Omega^-$ is the region enclosed. The gray region is the set $A_\epsilon(\Gamma\cup\alpha)$. }
\label{fig}
\end{figure}

\begin{lem}[Bridge Principle] \label{bridgee}
Let $\Gamma$ be as above and bound a unique absolutely area minimizing surface $\Sigma$. Then (with the notation above) there exists a sequence $t_n\in(0,\epsilon)$ converging to $0$ such that $\Gamma_{t_n} \subset \Omega^+\cap A_\epsilon(\Gamma\cup\alpha)$ bounds a unique absolutely area minimizing surface $\Sigma_n$ and that $\Sigma_n$ is homeomorphic to $\overline{\Sigma}\cup \overline{P_\alpha}$(closures taken in the closed ball $ \mathbb{H}^3 \cup \partial_\infty \mathbb{H}^3$).

\end{lem}

\section{Generic non-uniqueness of $H$-surfaces}\label{gener}

In this section we prove our main theorem. We will first prove the theorem for minimal surfaces. Then we will generalize to $0<|H|<1$. \smallskip
\subsection{The case $H=0$}
Let $A_0$ be the space of simple closed curves with $\text{C}^0$-topology (Definition \ref{nhood}) and let $A_1$ be the subset of $A_0$ consisting of curves for which any absolutely area minimizing surface has positive genus.
\begin{prop} \label{open}
$A_1$ is an open subset of the space of all simple closed curves.
\end{prop}
\begin{proof}
We will show that $A_0\setminus A_1$ is closed. Let $\Gamma_n$ be a sequence in $A_0\setminus A_1$ which converges to some $\Gamma\in A_0$ and $\Sigma_n$ be a sequence of absolutely area minimizing surfaces with $\partial_\infty \Sigma_n = \Gamma_n$ and that $\Sigma_n$ is homeomorphic to a disk  for each $n$. Then we can extract a subsequence of $\Sigma_n$ which converges smoothly on compact sets to an absolutely area minimizing surface $\Sigma$ with asymptotic boundary $\Gamma$. This follows from the existence of local area and curvature bounds (See Theorem 4.37 in \cite{MRR99}). In order to find local area bounds, let $\Omega_n$ be one of the connected open sets separated by $\Sigma_n$ in $\mathbb{H}^3$. Then for any compact domain $K$, since $\Sigma_n$ is absolutely area minimizing, $A(\Sigma_n\cap K)\leq A(\partial K \cap \Omega_n)$ where $A$ denotes the area. The latter has always less area than $\partial K$ so that $A(\partial K)$ serves as a local area bound. Also, by \cite{RST10}, there exists uniform global curvature bound for stable minimal surfaces in $\mathbb {H}^3$. 
Note also that since each $\Sigma_n$ is area minimizing, the convergence is with multiplicity $1$. From this it follows that for any compact domain $K$ and for a sufficiently large $n$, $\Sigma_n \cap K$ can be written as a normal graph over $\Sigma\cap K$ so that both surfaces are homeomorphic. Since this holds for any compact domain and $\Sigma_n$ is of genus $0$, $\Sigma$ must have genus $0$ too. Now we have $\partial_\infty \Sigma=\Gamma $ and $ \Gamma \in A_0\setminus A_1$, so the proof is complete.
\end{proof}

\noindent Next we show that $A_1$ is also dense.

\begin{prop}\label{dense}
$A_1$ is dense in the space of all simple closed curves.
\end{prop}

\begin{proof}

Let $\Gamma$ be a simple closed curve and $\epsilon>0$. By Lemma \ref{unique} there exists a smooth simple closed curve $\Gamma_u\in N_{\epsilon/2}(\Gamma)$ (Definition \ref{nhood})  which bounds a unique absolutely area minimizing surface $\Sigma$. Suppose that $\Gamma_u \notin A_1$, i.e. $\Sigma$ is a plane. Since $\Gamma_u$ is a uniquely minimizing curve we can apply the bridge principle. 

\noindent Let $\alpha$ be a line segment joining two distinct points of $\Gamma_u$ satisfying the hypothesis of Lemma \ref{bridgee} and such that $\Gamma_u\cup \alpha$ lies in the annulus $A_{\epsilon/2}(\Gamma_u)$ around $\Gamma_u$. Then by the bridge principle we can find a curve $\Gamma_u'\subset A_{\epsilon/2}(\Gamma_u)$ which is a disjoint union of two simple closed curves and bounds a unique absolutely area minimizing surface $\Sigma'$ (Figure \ref{fig:Bridge}). Furthermore $\overline{\Sigma'}$ is homeomorphic to $\overline\Sigma\cup \overline {P_\alpha}$ where $\overline {P_\alpha}$ is a \emph{bridge along} $\alpha$ (for the definition see the last paragraph in the previous section and also note that closures taken in the closed ball $ \mathbb{H}^3 \cup \partial_\infty \mathbb{H}^3$ ). Since $\Sigma$ is a plane, $\Sigma'$ is homeomorphic to an annulus. Now we can apply the bridge principle to $\Sigma'$ because it is also a uniquely area minimizing surface. 
\noindent Let $\alpha'$ be a line segment connecting two distinct points of $\Gamma_u'$ which lies on distinct components of $\Gamma_u'$ and such that $\Gamma_u'\cup \alpha'$ still lies inside $A_{\epsilon/2}(\Gamma_u)$. Then by the bridge principle again, we can find a new curve $\Gamma_u''\subset A_{\epsilon/2}(\Gamma_u)$ which is a simple closed curve which bounds a unique absolutely area minimizing surface $\Sigma''$. Furthermore, $\overline{\Sigma''}$ is homeomorphic to $\overline{\Sigma'}\cup \overline{P_{\alpha'}}$ where $P_{\alpha'}$ is a bridge along $\alpha'$. Since $\Sigma'$ is an annulus and the bridge $ \overline{P_{\alpha'}}$ connects distinct components of $\partial \overline {\Sigma'}$, $\Sigma''$ is homeomorphic to a surface of genus one with a disk removed. Since $\Gamma''_u$ is a simple closed curve lying inside $A_{\epsilon/2}(\Gamma_u)$ and homotopic to $\Gamma_u$, $\Gamma''_u$ belongs to $N_{\epsilon/2}(\Gamma_u)$. So it also belongs to $ N_{\epsilon}(\Gamma)$. Lastly, $\Gamma''_u$ bounds a unique absolutely area minimizing surface of genus one so that $\Gamma''_u\in A_1$. This completes the proof.

\end{proof}

\begin{figure}
\centering
\includegraphics[scale=0.75]{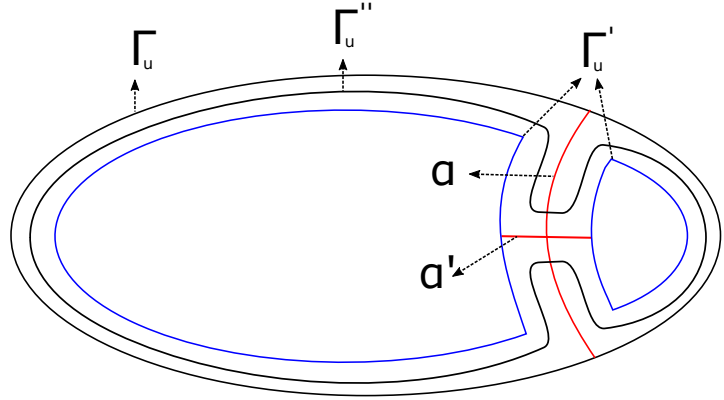}
\caption{We can apply the bridge principle successively to a uniquely minimizing curve $\Gamma_u$ successively in a neighborhood.}
\label{fig:Bridge}
\end{figure}

\begin{rem}\label{genus}
Let $g>0$ be an integer and let $A_g$ denote the set of all simple closed curves for which all absolutely area minimizing surfaces have genus greater than or equal to $g$. Then the previous proofs can be generalized to show that $A_g$ is open and dense for every $g\in \mathbb{N}$. This gives us the stratification of the space of simple closed curves as $A_0 \supset A_1 \supset A_2\supset ... \supset A_g  \supset... $ where $A_{g+1}$ is an open and dense subset of  $A_g$ for every $g\in \mathbb{N}$. 
\end{rem}

\begin{thm}\label{main}
The set of simple closed curves which bound more than one complete embedded minimal surface contains an open and dense subset of $A_0$.
\end{thm}

\begin{proof}

We know that every simple closed curve bounds a least area plane by Lemma \ref{plane}. From this it follows that for all $\Gamma\in A_1$ there exist at least two complete embedded minimal surfaces with asymptotic boundary $\Gamma$. Since $A_1$ is open and dense by the previous propositions, the result follows.

\end{proof}

\subsection{$H$-surfaces}\label{sectt}

In this section we generalize our result to $H$-surfaces, where $|H|<1$. Following the formulation of \cite{To96} and the terminology of \cite{Co05}, we first state the generalization of asymptotic Plateau problem for $H$-surfaces. For this we will use the upper half-space model for $ \mathbb{H}^3 =\{(x,y,z)\in  \mathbb{R}^3 | z>0\}$. So $\partial_\infty \mathbb{H}^3$ can be identified with the one-point compactification of $\mathbb{R}^2 \times \{0\}$.

\begin{defn}\label{harea}
Let $\Sigma\subset  \mathbb{H}^3$ be complete embedded surface with asymptotic boundary $\Gamma$. Then $\Gamma$ separates $\mathbb{R}^2 \times \{0\}$ into two domains one of which is bounded. Let $O\subset\mathbb{R}^2 \times \{0\}$ be the bounded part and $\Omega_\Sigma \subset \mathbb{H}^3 $ be the region enclosed by $\Sigma\cup O$. We say that $\Sigma$ is a $H$-\emph{minimizing} surface if for any compact domain $K\subset \mathbb{H}^3 $, $\Sigma$ minimizes the functional $I^H_K=A(\Sigma\cap K)+2HV(\Omega_\Sigma \cap K)$ where $A$ denotes the area and $V$ denotes the volume. We will also say that $\Sigma$ is \emph{uniquely} $H$-\emph{minimizing} if it is the unique $H$-minimizing surface with its asympotic boundary. Lastly, if $\Sigma$ is a complete embedded plane which minimizes the functional among all disks on each compact domain, then we will call $\Sigma$ a $H$-\emph{minimizing plane}.
\end{defn}
\begin{lem}[\cite{AR97}, \cite{To96}]
Let $\Gamma\subset \partial_\infty \mathbb{H}^3$ be $\textnormal{C}^1$ curve and $|H|<1$. Then there exists a complete embedded $H$-minimizing surface (with respect to the normal field pointing into  $\Omega_\Sigma$) with asymptotic boundary $\Gamma$.
\end{lem}

\begin{rem}
As noted in \cite{To96} it is enough to assume $O$ to be a set of locally finite perimeter which in turn is rectifiable. Any simple closed curve can be approximated by rectifiable sets which allows us to extend the theorem to simple closed curves. See also Remark \ref{simp}.
\end{rem}

\begin{cla*}
The previous theorem still holds if we take $\Gamma$ to be a simple closed curve.
\end{cla*}

\begin{proof}
Let $\Gamma$ be a simple closed curve. Then there exists a sequence of inscribed polygons $\Gamma_n$ which converges to $\Gamma$. By the above theorem there exists a sequence of $H$-minimizing surfaces $\Sigma_n$ with asymptotic boundary $\Gamma_n$. As we did in Proposition \ref{open} we can extract a subsequence of $\Sigma_n$ which converges smoothly on compact sets to an $H$-minimizing surface $\Sigma$ with asymptotic boundary $\Gamma$. For this it is enough to find uniform local area and curvature bounds. Let $K$ be a compact domain in $\mathbb{H}^3$. Then $A(\Sigma_n\cap K)\leq A(\partial K \cap \Omega_{\Sigma_n})$ because otherwise we could replace $\Sigma_n\cap K$ with $\partial K \cap \Omega_{\Sigma_n}$ decreasing the total value of the functional $I^H_K$ as the second term disappears. So $A(\partial K)$ serves as a local area bound. Moreover, there exists global curvature bound for stable complete embedded $H$-surfaces (\cite{RST10}).
\end{proof}

\noindent Recently, the asymptotic Plateau problem for $H$-minimizing disks is also obtained.

\begin{lem}[\cite{Co15}]\label{disk}
Given $|H| <1$ and any simple closed curve $\Gamma\subset \partial\mathbb{H}^3$ with at least one smooth point there exists an embedded minimizing $H$-plane $\Sigma$ with asymptotic boundary $\Gamma$.
\end{lem}

\noindent Because of the restriction of this theorem we will restrict ourselves to the space of simple closed curves which has at least one smooth point. So let $A^*_0$ be the space of simple closed curves with at least one smooth point and $A^H_1$ be the subset of $A^*_0$ consisting of curves for which any $H$-minimizing surface has positive genus.

\begin{prop}\label{op}
For each $|H|<1$, $A^H_1$ is an open and dense subset of $A^*_0$.
\end{prop}

\begin{proof}
The proof follows similarly as in the case of minimal surfaces. Let $\Gamma_n\in A_0\setminus A_1$ be a sequence converging to a curve $\Gamma\in A_0$. By definition, there exists $H$-minimizing surfaces $\Sigma_n$ which are planes and asymptotic to $\Gamma_n$. Then as in the proof of the previous claim we can extract a subsequence of $\Sigma_n$ which converges smoothly on compact sets to some $H$-minimizing plane $\Sigma$ with asymptotic boundary $\Gamma$. The limit surface $\Sigma$ is a plane which shows that $A_0\setminus A_1$ is closed, so $A^H_1$ is open. 

\noindent For the density we note that the bridge principle for absolutely area minimizing surfaces (Lemma \ref{bridgee}) applies to $H$-minimizing surfaces. Also, Lemma \ref{unique} also extends to uniquely minimizing $H$-surfaces (\cite{Co06}). Then the density of $A^H_1$ follows exactly the same way as in Proposition \ref{dense}.

\end{proof}

\begin{cor}

Given $|H|<1$, the set of simple closed curves (with at least one smooth point) which bounds more than one complete embedded $H$-surface contains an open and dense subset.

\end{cor}

\begin{proof}
By Lemma \ref{disk} every simple closed curve $\Gamma\in A^H_1$ bounds an $H$-minimizing plane so the result follows.
\end{proof} \smallskip
\section{Concluding Remarks}

In this section we will comment on two natural extensions of our result. Firstly, we will ask what happens for the space of  $\text{C}^1$ simple closed curves. Then we consider the case for curves with more than one component.

\subsection{The space of  $\text{C}^1$ simple closed curves}
\noindent We have obtained the generic non-uniqueness result for the space of simple closed curves endowed with the $\text{C}^0$-topology. A natural question is what happens for  $\text{C}^1$ simple closed curves? We know that a simple closed curve $\Gamma\subset \partial_\infty\mathbb{H}^3$ which bounds a star-shaped domain bounds a unique complete embedded minimal surface which in turn has to be a plane (\cite{HL87}). Following lemma allows us to conclude that there exist $\text{C}^1$-open neighborhoods consisting of simple closed curves which bound a unique complete embedded minimal surface. \smallskip

\noindent We consider the upper half-space model for $\mathbb{H}^3$ (see the beginning of Section \ref{sectt}). We will use the following characterization of star-shaped domains: A bounded simply-connected domain $\Omega \subset \mathbb{R}^2$ with a $\text{C}^1$ boundary is star-shaped with respect to a point $x \in \Omega$ if and only if $(y-x)\cdot \nu(y) \leq 0$ for all $y\in\partial \Omega$ where $\nu$ is the inward pointing normal on $\partial \Omega$.

\begin{lem}\label{star}
Let $\Gamma\subset \mathbb{R}^2$ be a convex simple closed curve. Then there exists a $\textnormal{C}^1$-neighbourhood of $\Gamma$ consisting only of curves which bound a star-shaped domain.

\end{lem}

\begin{proof}

Let $\Gamma$ be a convex curve such that the origin is contained in the bounded region enclosed by $\Gamma$. We claim that in a sufficiently small $\text{C}^1$-neighborhood of $\Gamma$ every curve bounds a star-shaped region with respect to the origin. Otherwise, there exists a sequence of curves $\Gamma_n$ which are not star-shaped with respect to the origin and the sequence converges to $\Gamma$ in $\textnormal{C}^1$. Then, by the characterization of star-shaped domains given above, there exists a sequence of points $x_n \in \Gamma_n$ with $x_n\cdot \nu_n(x_n) \geq 0$ where $\nu_n$ is the inward pointing normal vector field to $\Gamma_n$. By passing to a subsequence we can assume that $x_n$ converges to a point $x\in\Gamma$ and the tangent lines at $x_n$ to $\Gamma_n$ converge to the tangent line of $\Gamma$ passing through $x$. In particular $\nu_n(x_n)$ converges to $\nu(x)$. Hence, we find a point $x\in\Gamma$ such that $x\cdot \nu(x) \geq 0$. On the other hand since $\Gamma$ is also star-shaped with respect to the origin, we have $x\cdot \nu(x) \leq 0$ so that $x\cdot \nu(x) = 0$. Therefore the tangent line $L$ of $\Gamma$ passing through $x$ coincides with the line passing through the origin and $x$. Hence, $L$ intersects the bounded open set enclosed by $\Gamma$ which contradicts with the convexity of $\Gamma$.
\end{proof} \smallskip

By the above lemma and the uniqueness results for curves which bound star-shaped domains (\cite{HL87}) we have the following corollary.

\begin{cor}
There exist $\textnormal{C}^1$-open subsets of the space of simple closed curves in which every curve bounds a unique complete embedded minimal surface. 
\end{cor}

\noindent Hence, generic non-uniqueness does not hold for the space of $\text{C}^1$ simple closed curves.

\subsection{Curves with more than one component}

\noindent Given a closed curve $\Gamma\subset  \partial \mathbb{H}^3$ with more than one component, in general there is no connected minimal surface with asymptotic boundary $\Gamma$. For instance, if $\Gamma$ is a disjoint union of two round circles and if $\Sigma$ is a connected minimal surface, then $\Sigma$ has to be a surface of revolution (\cite{LR85}). However, there exist pairs of round circles which cannot bound a surface of revolution (\cite{Wa15}). On the other hand solving the asymptotic Plateau problem in mean convex domains of $\mathbb{H}^3$ (\cite{OS98}) we can obtain plenty of connected minimal surfaces with two boundary components. Then it might be possible to apply the degree theory developed for minimal surfaces in $\mathbb{H}^3$ (\cite{AM10}) in order to obtain non-uniqueness results in this case.

Of course, if we don't restrict to connected minimal surfaces, existence problem is still solvable by absolutely area minimizing surfaces. Then we can apply the bridge principle as we did before to show that $A_1$ is an open and dense subset of $A_0$. Hence, generic non-uniqueness follows as any curve with more than one component bounds a disjoint union of least area planes. More precisely the following generalization of our main theorem holds.

\begin{thm}
Let $A_{0,k}$ be the space of curves with $k$ components. Then the subset of $A_{0,k}$ consisting of curves which bound more than one complete embedded minimal surface (not necessarily connected) contains an open and dense subset of $A_{0,k}$.
\end{thm}
\bibliographystyle{siam}
\bibliography{CagriHaciyusufoglu}

\begin{thebibliography}{10}

\bibitem{AR97}
{\sc H.~Alencar and H.~Rosenberg}, {\em Some remarks on the existence of
  hypersurfaces of constant mean curvature with a given boundary, or asymptotic
  boundary, in hyperbolic space}, Bull. Sci. Math., 121 (1997), pp.~61--69.

\bibitem{AM10}
{\sc S.~Alexakis and R.~Mazzeo}, {\em Renormalized area and properly embedded
  minimal surfaces in hyperbolic 3-manifolds}, Comm. Math. Phys., 297 (2010),
  pp.~621--651.

\bibitem{An82}
{\sc M.~T. Anderson}, {\em Complete minimal varieties in hyperbolic space},
  Invent. Math., 69 (1982), pp.~477--494.

\bibitem{An83}
\leavevmode\vrule height 2pt depth -1.6pt width 23pt, {\em Complete minimal
  hypersurfaces in hyperbolic {$n$}-manifolds}, Comment. Math. Helv., 58
  (1983), pp.~264--290.

\bibitem{TC15}
{\sc T.~Bourni and B.~Coskunuzer}, {\em Area minimizing surfaces in mean convex
  3-manifolds}, J. Reine Angew. Math., 704 (2015), pp.~135--167.

\bibitem{Co15}
{\sc B.~Coskunuzer}, {\em Asymptotic h-plateau problem in hyperbolic 3-space}.
\newblock arxiv/1505.00650, (2015).

\bibitem{Co13-1}
\leavevmode\vrule height 2pt depth -1.6pt width 23pt, {\em H-surfaces with
  arbitrary topology in hyperbolic 3-space}.
\newblock arxiv/1311.4629, (2013).

\bibitem{Co06}
{\sc B.~Coskunuzer}, {\em Generic uniqueness of least area planes in hyperbolic
  space}, Geom. Topol., 10 (2006), pp.~401--412.

\bibitem{Co05}
\leavevmode\vrule height 2pt depth -1.6pt width 23pt, {\em Minimizing constant
  mean curvature hypersurfaces in hyperbolic space}, Geom. Dedicata, 118
  (2006), pp.~157--171.

\bibitem{Co11}
{\sc B.~Coskunuzer}, {\em On the number of solutions to asymptotic plateau
  problem}, J. Gokova Geom. Topol., 5 (2011), pp.~1--19.

\bibitem{Co13-2}
\leavevmode\vrule height 2pt depth -1.6pt width 23pt, {\em Asymptotic plateau
  problem}, Proc. Gokova Geom. Top., Conf. (2013), pp.~120--146.

\bibitem{OS98}
{\sc G.~de~Oliveira and M.~Soret}, {\em Complete minimal surfaces in hyperbolic
  space}, Math. Ann., 311 (1998), pp.~397--419.

\bibitem{Ga97}
{\sc D.~Gabai}, {\em On the geometric and topological rigidity of hyperbolic
  {$3$}-manifolds}, J. Amer. Math. Soc., 10 (1997), pp.~37--74.

\bibitem{HL87}
{\sc R.~Hardt and F.-H. Lin}, {\em Existence and regularity of constant mean
  curvature hypersurfaces in hyperbolic space}, Invent. Math., 88 (1987),
  pp.~217--224.

\bibitem{Ha92}
{\sc J.~Hass}, {\em Intersections of least area surfaces}, Pacific J. Math.,
  152 (1992), pp.~119--123.

\bibitem{LR85}
{\sc G.~Levitt and H.~Rosenberg}, {\em Symmetry of constant mean curvature
  hypersurfaces in hyperbolic space}, Duke Math. J., 52 (1985), pp.~53--59.

\bibitem{MW13}
{\sc F.~Martin and B.~White}, {\em Properly embedded, area-minimizing surfaces
  in hyperbolic $3$-space},  (2013).

\bibitem{MRR99}
{\sc W.~H. Meeks, III, A.~Ros, and H.~Rosenberg}, {\em The global theory of
  minimal surfaces in flat spaces}, vol.~1775 of Lecture Notes in Mathematics,
  Springer-Verlag, Berlin; Centro Internazionale Matematico Estivo (C.I.M.E.),
  Florence, 2002.

\bibitem{RST10}
{\sc H.~Rosenberg, R.~Souam, and E.~Toubiana}, {\em General curvature estimates
  for stable {$H$}-surfaces in 3-manifolds and applications}, J. Differential
  Geom., 84 (2010), pp.~623--648.

\bibitem{To96}
{\sc Y.~Tonegawa}, {\em Existence and regularity of constant mean curvature
  hypersurfaces in hyperbolic space}, Math. Z., 221 (1996), pp.~591--615.

\bibitem{Wa15}
{\sc B.~Wang}, {\em Least area spherical catenoids in hyperbolic
  three-dimensional space}.
\newblock arxiv/1204.4943, (2012).

\end{thebibliography}

\end{document}